\documentclass[psamsfonts]{amsart}

\usepackage{amsmath, amsthm, amsfonts, amssymb, graphicx, tikz, url}
\usepackage{fifteen-figures}
\usetikzlibrary{arrows}
\usetikzlibrary{matrix}

\makeatletter 
\newtheorem*{rep@theorem}{\rep@title}
\newcommand{\newreptheorem}[2]{%
\newenvironment{rep#1}[1]{%
 \def\rep@title{#2 \ref{##1}}%
 \begin{rep@theorem}}%
 {\end{rep@theorem}}}
\makeatother

\newtheorem{theorem}{Theorem}
\newreptheorem{theorem}{Lemma}
\newtheorem{lemma}[theorem]{Lemma}
\newtheorem{sublemma}[theorem]{Sublemma}

\newtheorem{proposition}[theorem]{Proposition}

\theoremstyle{definition}

\DeclareMathOperator{\sq}{sq}
\DeclareMathOperator{\hex}{hex}
\DeclareMathOperator{\diam}{diam}

\begin{document}

\title[Square and hexagon configurations]{Discrete configuration spaces\\ of squares and hexagons}
\author{Hannah Alpert}
\address{231 West 18th Avenue, Columbus, OH 43210}
\email{alpert.19@osu.edu}
\begin{abstract}
We consider generalizations of the familiar fifteen-piece sliding puzzle on the $4$ by $4$ square grid.  On larger grids with more pieces and more holes, asymptotically how fast can we move the puzzle into the solved state?  We also give a variation with sliding hexagons.  The square puzzles and the hexagon puzzles are both discrete versions of configuration spaces of disks, which are of interest in statistical mechanics and topological robotics.  The combinatorial theorems and proofs in this paper suggest followup questions in both combinatorics and topology, and may turn out to be useful for proving topological statements about configuration spaces.
\end{abstract}

\maketitle

\section{Introduction}
The fifteen-piece sliding puzzle is familiar from childhood: a $4$ by $4$ square grid with fifteen distinct square pieces and one square hole.  By sliding the pieces we can achieve any even permutation of the fifteen pieces.   In this paper we consider puzzles on larger grids with more pieces and more holes.  Asymptotically, as the grids grow, how fast can we move the puzzle into the solved state?  Theorem~\ref{sq-diam} gives one answer to this question and is the main theorem of this paper.

This question about generalizations of the fifteen-piece puzzle is part of a larger body of work about configuration spaces of disks and spheres.  If we consider the space of configurations of $n$ distinct round particles of radius $r$ in a box of side length $1$, how do the properties of this space change as we fix $n$ and vary $r$, or as we increase $n$ and decrease $r$ as a function of $n$?  In physics the hope is that topological and geometric properties of the configuration space can be used to predict physical phenomena such as phase transition; the paper~\cite{Carlsson12} describes contributions to the literature, including \cite{Adler62, Teixeira04, Grinza04, Angelani05, Farber10, Franzosi07, Franzosi07b, Kastner07, Kastner08, Caiani97, Franzosi99, Casetti99, Casetti02, Franzosi00}.  A purely topological exploration of these configuration spaces can be found in papers such as \cite{Deeley11, Baryshnikov13, Alpert16}.  Similar questions are also part of topological robotics \cite{Abrams02, Farber08, Ghrist10, LaValle06}.  The sliding puzzles in the present paper can be seen as discrete approximations of the disk configuration spaces; presumably, many properties of the discrete and continuous versions ought to correspond.  Because of this context, the combinatorial results in this paper have both combinatorial and topological implications.  On the one hand, it would be interesting to generalize these results in a combinatorial direction; on the other, the results suggest analogous questions and proof methods applicable to the disk configuration spaces, and thus should be of interest to topologists.

One property of sliding disks that is not true of sliding squares is the ability to get stuck: there are arbitrarily sparse configurations of disks in which the disks cannot move (see \cite{Boroczky64} and \cite{Kahle12}).  In order to get closer to the behavior of disks, in this paper we also consider sliding hexagons.  However, although the hexagons are harder to move than squares, in some sense the difficulty occurs only up to a constant factor; this is the content of Theorem~\ref{hex-diam}, the hexagon analogue of Theorem~\ref{sq-diam} for squares.

In the remainder of this introductory section we define the discrete configuration spaces of squares.  Let $\widetilde{\sq}(n; m)$ denote the set of ways to arrange $n$ pieces numbered $1$ through $n$ in distinct positions on the $m$ by $m$ square grid (so we have $n \leq m^2$).  We refer to this as a \textbf{\textit{labeled configuration space}}, and let the \textbf{\textit{unlabeled configuration space}} $\sq(n; m)$ be the set of ways to arrange $n$ indistinguishable pieces in distinct positions on the $m$ by $m$ square grid.  Sometimes instead of an $m$ by $m$ square grid we want an $m_1$ by $m_2$ rectangular subset of the square grid, and use the notation $\widetilde{\sq}(n; m_1, m_2)$ and $\sq(n; m_1, m_2)$ for the corresponding configuration spaces.

These configuration spaces can be viewed as metric spaces.  First we give a naive choice of metric and show in Proposition~\ref{naive} that we can easily estimate, up to a constant factor, the diameter of the resulting metric space.  Then we give a more physically relevant metric that we use for the remainder of the paper.

Naively, we can say that two configurations are adjacent, or have distance $1$, if they differ by moving one piece into a hole (i.e., an unoccupied position) adjacent to it.  We refer to the metric generated by this relation as the \textbf{\textit{$L^1$ intrinsic metric}}.  It is always greater than or equal to the \textbf{\textit{$L^1$ extrinsic metric}}, under which the distance between two configurations is the sum over all pieces of the distance along the grid between the two positions of that piece.  (This definition is for the labeled configuration space; for the unlabeled space, take the minimum distance achieved over all labelings of the two configurations.)  That is, the $L^1$ extrinsic metric would be the minimum number of moves between the two configurations if multiple pieces were allowed to occupy the same position.

We would like to estimate the diameter of the labeled configuration space $\widetilde{\sq}(n; m)$.  As $m$ grows, does the $L^1$ intrinsic diameter grow a lot faster than the $L^1$ extrinsic diameter?  The following proposition shows that the answer is no.

\begin{proposition}\label{naive}
The labeled configuration space $\widetilde{\sq}(n; m)$, with $n \leq m^2 - 2$ so that the space is connected, has $L^1$ intrinsic diameter that grows as $\Theta(nm)$, matching the $L^1$ extrinsic diameter up to a constant factor.
\end{proposition}

\begin{proof}
We check first that the $L^1$ extrinsic diameter grows as $\Omega(nm)$, and then that the $L^1$ intrinsic diameter grows as $O(nm)$.  For the extrinsic: take any configuration, and swap the top half and bottom half.  Each piece travels distance $\left\lfloor \frac{m}{2} \right \rfloor$ or $\left\lceil \frac{m}{2} \right\rceil$, so the $L^1$ extrinsic distance between the two configurations is at least $\left \lfloor \frac{m}{2} \right \rfloor \cdot n$.

To check that the $L^1$ intrinsic diameter grows as $O(nm)$, we designate one home configuration and provide an $O(nm)$--step algorithm for getting home.  The home configuration is as follows, shown in Figure~\ref{fig-sq-induction}: numbers $1$ through $m$ go in order along the top row, then $m+1$ through $2m-1$ down the left column, then the remainder of the second row, then the remainder of the second column, and so on until the pieces run out.  We use induction on $m$.  It takes $O(m)$ moves to move a hole to be adjacent to piece $1$ and then move piece $1$ home, then another $O(m)$ moves to bring piece $2$ home, and so on, to fill the top row and left column with the correct pieces.  Then we apply the inductive hypothesis to the remaining grid.
\end{proof}

\begin{figure}
\begin{center}
\SqInduction
\end{center}
\caption{The home configuration for $n = 14$ pieces on a board of side $m = 4$.  Because we can move the pieces home in numerical order in $O(nm)$ moves, the $L^1$ intrinsic diameter of the configuration space $\widetilde{\sq}(n; m)$ is $O(nm)$.}\label{fig-sq-induction}
\end{figure}

Having quickly resolved the question of $L^1$ intrinsic diameter, we do not pursue it any further.  In the remainder of the paper, we allow arbitrarily many independent moves to occur simultaneously.  This new construction is more reasonable from a physics perspective, if we imagine the pieces to be independent particles that might randomly wander into neighboring holes.  We say that the \textbf{\textit{$L^\infty$ intrinsic metric}} is generated by the relation that two configurations are adjacent, and have distance $1$, if they differ by moving any number of pieces into holes adjacent to them.  (Note that we do not permit sliding a piece into a position that is simultaneously being vacated; it must really be a hole.)  The \textbf{\textit{$L^\infty$ extrinsic metric}} is the maximum over all pieces of the distance along the grid between the positions of that piece in the two configurations.  Of course, if the number of holes is bounded, then the $L^\infty$ intrinsic metric and the $L^1$ intrinsic metric differ by at most a constant factor.  In the remainder of the paper, we fix the proportion of positions that contain pieces, so that as $m$ grows the number of holes grows.

In Section~\ref{sec-sq} we prove the main theorem for squares (Theorem~\ref{sq-diam}), which says that for a fixed ratio of pieces to holes, the $L^\infty$ intrinsic and extrinsic diameters match up to a constant factor.  Sections~\ref{sec-hex} and~\ref{sec-hole} address the analogous result, the main theorem for hexagons (Theorem~\ref{hex-diam}).  In Section~\ref{sec-hex} we define the discrete configuration spaces of hexagons and prove basic connectivity properties, and in Section~\ref{sec-hole} we prove the hole-moving lemma for hexagons (Lemma~\ref{hex-hole}), which serves to reduce the hexagon problem to the square problem, completing the proof of the main theorem for hexagons.

\emph{Acknowledgments.}  I would like to thank Matt Kahle for suggesting the problem of estimating the diameters of square sliding-piece puzzles, and Larry Guth and Jeremy Mason for additional helpful discussions.  I was supported by the Institute for Computational and Experimental Research in Mathematics (ICERM) for the full duration of this project.

\section{Main theorem for squares}\label{sec-sq}

In this section we prove the main theorem for squares (Theorem~\ref{sq-diam}), except for the proof of the hole-moving lemma for squares (Lemma~\ref{sq-hole}).  The proof of this lemma appears in Section~\ref{sec-hole}, even though it is not hard, because it serves as a template for the proof of the analogous lemma about hexagons, which is more involved.

\begin{theorem}[Main theorem for squares]\label{sq-diam}
Fix $\alpha \in (0, 1)$.  Then for $m$ a power of $2$, and $\frac{\alpha}{2}m^2 < n \leq \alpha m^2$, the labeled configuration space $\widetilde{\sq}(n; m)$ has $L^\infty$ intrinsic diameter that grows as $\Theta(m)$, matching the $L^\infty$ extrinsic diameter up to a constant factor depending only on $\alpha$.
\end{theorem}

The proof of the theorem is based on the solution to the following simpler problem.  Suppose that we have an $m$ by $m$ grid with pieces numbered $1$ through $m^2$ in distinct positions.  We can swap a pair of adjacent pieces, and in fact any disjoint set of such swaps can happen simultaneously.  How long does it take to sort the pieces?  Taking the $L^\infty$ extrinsic diameter gives a lower bound of $\Omega(m)$ time steps; can we sort in $\Theta(m)$ time steps?

An affirmative answer was found by Thompson and Kung in 1977, and appears below as Theorem~\ref{grid-swaps}.  Furthermore, the sequence of comparisons in their algorithm does not depend on the starting permutation; such an algorithm is called a \textbf{\textit{sorting network}}.  More specifically, by a \textbf{\textit{comparison step}} we mean an ordered pair of positions adjacent in the grid, indicating that the pieces in those positions should be swapped if necessary to match the order specified.  By a \textbf{\textit{routing step}} we mean an unordered pair of positions adjacent in the grid, indicating that the pieces in those positions should be swapped without comparing them.  The algorithm is a sequence of time steps, each consisting of a collection of disjoint comparison steps or routing steps.  The pieces end up in \textbf{\textit{snake-like row major order}}, which means that the numbering goes across the top row to the right, across the second row to the left, and so on, alternating directions.

\begin{theorem}[\cite{Thompson77}]\label{grid-swaps}
Consider the $m$ by $m$ grid graph, with pieces numbered $1$ through $m^2$ distributed one per vertex.  For $m$ a power of $2$, there is a sorting-network algorithm to sort the pieces into snake-like row major order, using $\Theta(m)$ time steps of disjoint comparisons or routings along grid edges.
\end{theorem}

Given a sorting network on $n$ numbered pieces, we can apply it instead to $n$ buckets each of $k$ numbered pieces, as follows.  At each comparison step, we take the $2k$ pieces in the two buckets and redistribute so that the least $k$ pieces are in one bucket and the greatest $k$ pieces are in the other.  Then the greater of the two buckets takes the role of the greater of the two pieces in the original sorting network, and the lesser bucket takes the role of the lesser piece.  I would not be surprised if the following lemma is well-known, but the proof is included for completeness.

\begin{lemma}[Bucket lemma]\label{bucket}
Given a sorting network on $n$ pieces, if we apply it instead to $n$ buckets of $k$ pieces, then it sorts the pieces: the least $k$ pieces end up in the least bucket, the next $k$ pieces in the next bucket, and so on.
\end{lemma}

\begin{proof}
We use the 0-1 principle for sorting networks, explained on p.~224 of \cite{Knuth73}, which says that it suffices to show that if all pieces have labels 0 and 1 instead of distinct labels, then the algorithm sorts them so that all 0 pieces precede all 1 pieces.  We use induction on the number of buckets in the sorting configuration that contain both 0 and 1 (mixed buckets).  The base case is if there is at most one mixed bucket; in this case the sorting proceeds as if each bucket were a 0 piece, a 1 piece, or in the case of the mixed bucket, a piece labeled $\frac{1}{2}$.  These buckets are sorted correctly by the sorting network.

Suppose that there are at least two mixed buckets in the starting configuration.  There must be, at some time during the sorting, a comparison step involving two mixed buckets; we select a first such comparison step.  After this step, at most one of the two buckets is a mixed bucket, and the sorting proceeds the same as if these two buckets had come into this comparison step in the same state that they leave it.  That is, we can follow these two buckets backward in time to construct a different starting configuration with one mixed bucket fewer than our original starting configuration, but with the same ending configuration.  By the inductive hypothesis, this ending configuration must be the sorted configuration.
\end{proof}

The remaining ingredient needed to prove the main theorem for squares (Theorem~\ref{sq-diam}) is the following lemma, for which the proof appears in Section~\ref{sec-hole} along with the corresponding proof for hexagons.

\begin{lemma}[Hole-moving lemma for squares]\label{sq-hole}
For $m$ a power of $2$, and arbitrary $n \leq m^2$, the unlabeled configuration space $\sq(n; m)$ has $L^\infty$ intrinsic diameter that grows as $O(m)$ independent of $n$.
\end{lemma}

Using this lemma we can finish the proof of the main theorem for squares.

\begin{proof}[Proof of main theorem for squares (Theorem~\ref{sq-diam})]
We need to provide an algorithm for sorting $\widetilde{\sq}(n; m)$ to some home configuration in $O(m)$ time steps.  Roughly, for some choice of $r$ we view the $m$ by $m$ grid as an $\frac{m}{r}$ by $\frac{m}{r}$ grid of $r$ by $r$ blocks, and each block plays the role of a bucket in the sorting network on the $\frac{m}{r}$ by $\frac{m}{r}$ grid.

More specifically, we fix $r$ a power of $2$ large enough that $\alpha < \frac{r^2 - 1}{r^2}$.  This means that if we view the $m$ by $m$ grid as an $\frac{m}{r}$ by $\frac{m}{r}$ grid of $r$ by $r$ blocks, then we have enough holes to put at least one in each block.  Note that the diameter is monotonic in $n$, since we can always delete a piece from any sequence of moves.  So it suffices to provide a sorting algorithm for the case where $n = \left(\frac{m}{r}\right)^2(r^2 - 1)$ and we have the same number of holes as blocks.

The sorting algorithm goes as follows.  First we use the hole-moving lemma (Lemma~\ref{sq-hole}) to move the holes so that each block has one hole.  Then we apply Thompson and Kung's sorting algorithm (Theorem~\ref{grid-swaps}), using the blocks as the buckets from Lemma~\ref{bucket}.  That is, each pair of blocks forms an $r$ by $2r$ or $2r$ by $r$ grid with two holes, so it can be sorted so that the lesser half of the pieces end up in the first block in order, and the greater half end up in the second block in order; the time required is a function only of $r$ (and hence of $\alpha$ but not $m$).  All together, the sorting takes $O\left(\frac{m}{r}\right) = O(m)$ time steps.

At the end of this process, the blocks are in snake-like row major order and the pieces in each block are sorted---say, also in snake-like row major order with the hole at the end.  Thus, it takes $O(m)$ time steps to get from an arbitrary configuration to this home configuration.
\end{proof}

\section{Sorting hexagons with six holes}\label{sec-hex}

In the hexagonal sliding puzzle, instead of numbered square tiles we use numbered hexagon tiles.  Geometrically if a hole is surrounded by pieces, then none of the pieces can move into the hole.  But if two holes are adjacent, then a piece adjacent to both holes can slide into either hole, as shown in Figure~\ref{fig-hex-between}.  We can ask the same asymptotic questions about the diameter of this hexagon puzzle as for the square puzzle.  For our growing boards it would be most natural to use a hexagon shape, but we use a parallelogram shape so that the techniques from the square case carry over better.  By an \textbf{\textit{$m_1$ by $m_2$ parallelogram}} we mean $m_1$ rows each of $m_2$ hexagons, each row half a step to the right of the previous row.

\begin{figure}
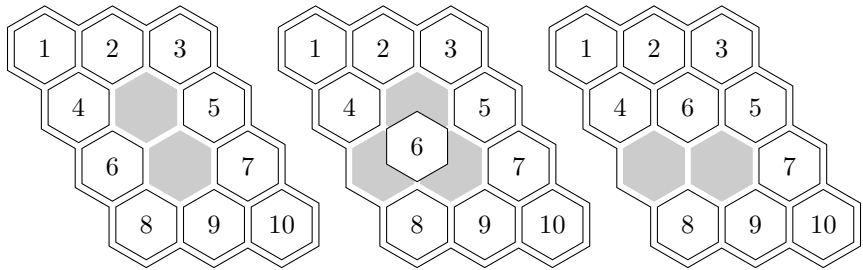

\begin{center}
\HexBetween
\end{center}
\caption{In the configuration space $\widetilde{\hex}(10; 4, 3)$ a piece adjacent to two adjacent holes can move into either of those positions.  The intermediate configuration, with the piece balanced between three positions, is not part of the configuration space.}\label{fig-hex-between}
\end{figure}

We let $\widetilde{\hex}(n; m_1, m_2)$ and $\hex(n; m_1, m_2)$ denote the labeled and unlabeled configuration spaces of $n$ hexagon pieces on an $m_1$ by $m_2$ parallelogram board, and let $\widetilde{\hex}(n; m)$ and $\hex(n; m)$ denote those on an $m$ by $m$ parallelogram board.  We allow a piece to move into either of two adjacent holes, and require that if multiple pieces move simultaneously, the pairs of holes that they use must be disjoint.  If we can make a set of simultaneous moves between two configurations in this way, we say that the configurations are adjacent and have $L^\infty$ intrinsic distance $1$; the \textbf{\textit{$L^\infty$ intrinsic metric}} is generated by these pairs.  The \textbf{\textit{$L^\infty$ extrinsic metric}} is, as with squares, the maximum over all pieces of the number of steps needed to move the piece on an empty board between its two positions.

We would like to prove a theorem for hexagons analogous to the main theorem for squares (Theorem~\ref{sq-diam}).  However, unlike the square configuration spaces with at least two holes, the space $\widetilde{\hex}(n; m)$ is not always connected.  Therefore it does not make sense a priori to ask about its diameter.  For instance, any configuration where all holes are isolated is not adjacent to any other configuration.  However, the next proposition states that as long as there are at least six holes, this is the only obstruction to connectivity, so each hexagon configuration space with at least six holes consists of a large connected component and possibly some isolated configurations.  Just as sorting a square grid with two holes is an important step in the main theorem for squares (Theorem~\ref{sq-diam}), this proposition is an important step in the main theorem for hexagons (Theorem~\ref{hex-diam}), which concerns the diameter of the large connected component of $\widetilde{\hex}(n; m)$.

\begin{proposition}\label{hex-connect}
For $m \geq 5$ and $n \leq m^2 - 6$, in the labeled configuration space $\widetilde{\hex}(n; m)$ there is a large connected component containing every configuration for which not all the holes are isolated.
\end{proposition}

\begin{proof}
It suffices to prove the statement when there are exactly six holes.  We start with an arbitrary configuration in which some two holes are adjacent, and show that we can move the holes into a $3$ by $2$ parallelogram in the upper-left corner, and sort the pieces into an arbitrary order.

First we move the holes.  From the point of view of the holes, a hole is permitted to move one step whenever there is another hole to which it remains adjacent.  So, we can take the two adjacent holes in our starting configuration and walk them over to a third hole, then move all three---staying connected---into the upper left corner.  Then, depositing one hole in the corner, we can walk the remaining two over to retrieve the next hole, and so on.  In this way we move the holes to form a $3$ by $2$ parallelogram in the upper-left corner.

Next we want to show that we can use this block of holes to permute the pieces arbitrarily.  We start by showing this for a $3$ by $4$ board: the holes are in the first two columns, and pieces labeled $1$ through $6$ are in the remaining two columns, say in column-major order.  We claim that we can transpose pieces $1$ and $2$, pieces $2$ and $3$, pieces $3$ and $4$, pieces $4$ and $5$, or pieces $5$ and $6$.  By composing these transpositions we can achieve any permutation.  The method for making the transpositions is as follows, shown in Figure~\ref{fig-hex-swap}.  It is not hard to see how to transpose pieces $1$ and $2$, or $2$ and $3$, while keeping pieces $4$, $5$, and $6$ in place.  Similarly, to transpose pieces $4$ and $5$, or $5$ and $6$, we start by shifting all pieces left by two columns.  It remains to see how to transpose pieces $3$ and $4$.  There is enough room to do this if we start with pieces $1$, $2$, and $3$ shifted to the leftmost column and pieces $4$, $5$, and $6$ in the rightmost column.

\begin{figure}
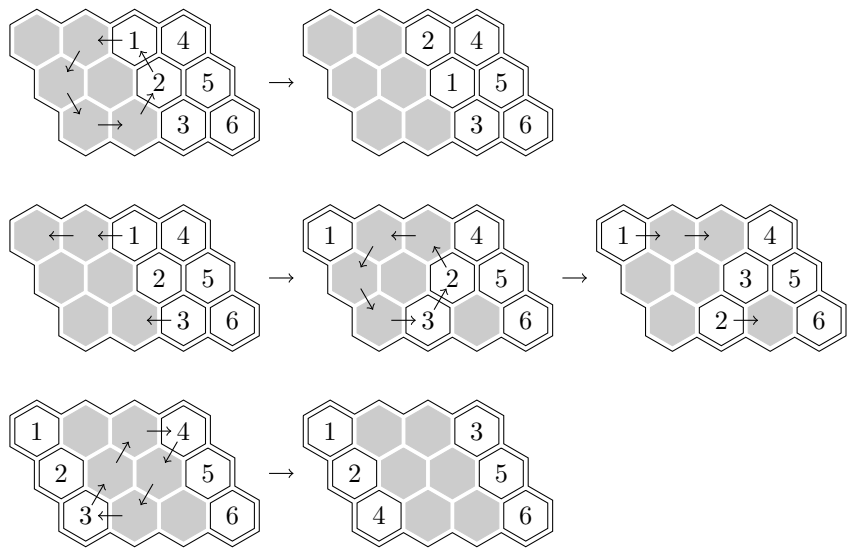

\begin{center}
\HexSwap
\end{center}
\caption{In this $3$ by $4$ parallelogram with $6$ pieces, we can swap any two consecutively numbered pieces; swaps $(12)$, $(23)$, and $(34)$ are shown.}\label{fig-hex-swap}
\end{figure}

This same method shows that we can sort a $3$ by $m$ parallelogram for arbitrary $m$, with the first two columns all holes.  For a parallelogram with more rows, we show that we can transpose any two adjacent pieces; these transpositions are more than enough to generate all permutations.  Given such a pair of adjacent pieces, if they are both in the first three rows, then we already know what to do.  Otherwise, find two adjacent columns that do not contain either of the pieces---this is why we have assumed $m \geq 5$---and shift the block of holes to the right until it is in that pair of columns.  Then shift the block of holes down until the pair of pieces we want to transpose is in the three rows that have holes in them.  We use the $3$ by $m$ case to achieve the desired transposition, then undo the hole-moving to return the other pieces to where they were.  By composing such transpositions, we can sort the pieces into any order.

Note that the proposition is still true for the cases $m = 3$ and $m = 4$, which can be done by hand, but the proof is not included here.
\end{proof}

Having established the existence of a large connected component, we are ready to state the main theorem for hexagons.

\begin{theorem}[Main theorem for hexagons]\label{hex-diam}
Fix $\alpha \in (0, 1)$.  Then for $m$ a power of $2$, and $\frac{\alpha}{2}m^2 < n \leq \alpha m^2$, the large connected component of the labeled configuration space $\widetilde{\hex}(n; m)$ has $L^\infty$ intrinsic diameter that grows as $\Theta(m)$, matching the $L^\infty$ extrinsic diameter up to a constant factor depending only on $\alpha$.
\end{theorem}

The proof is completely analogous to the proof of the main theorem for squares (Theorem~\ref{sq-diam}).  However, the hole-moving step, in which the holes are moved from their arbitrary starting positions to be evenly distributed over the board, is trickier for hexagons.  The hole-moving lemma (Lemma~\ref{hex-hole}) is the main goal of Section~\ref{sec-hole}, which then ends with the completed proof of the main theorem for hexagons.

\section{Hole-moving lemmas}\label{sec-hole}

In this section we prove the hole-moving lemma for squares (Lemma~\ref{sq-hole}), the hole-moving lemma for hexagons (Lemma~\ref{hex-hole}), and the main theorem for hexagons (Theorem~\ref{hex-diam}).  First we prove the hole-moving lemma for squares as a template for the hole-moving lemma for hexagons.  (There are simpler proofs for the square version, but we give a proof here that can be modified for hexagons.)  Then we prove the hexagon version and finish the main theorem for hexagons.

Here we repeat the statement of the hole-moving lemma for squares.

\begin{reptheorem}{sq-hole}[Hole-moving lemma for squares]
For $m$ a power of $2$, and arbitrary $n \leq m^2$, the unlabeled configuration space $\sq(n; m)$ has $L^\infty$ intrinsic diameter that grows as $O(m)$ independent of $n$.
\end{reptheorem}

The proof relies on the hole-pushing sublemma (Sublemma~\ref{sq-push}) and the hole-turning sublemma (Sublemma~\ref{sq-turn}).

\begin{sublemma}[Hole-pushing sublemma for squares]\label{sq-push}
For arbitrary $n \leq m$, in the one-row unlabeled configuration space $\sq(n; 1, m)$ the $L^\infty$ intrinsic distance between the configuration with all holes on the left and the configuration with all holes on the right is bounded above by $m$.
\end{sublemma}

\begin{proof}
Start with all holes on the left, and number the pieces $1$ through $n$ from left to right.  We can move piece $k$ to the left at time steps $k$ through $k + m - n - 1$ to get to the configuration with all holes on the right.
\end{proof}

\begin{sublemma}[Hole-turning sublemma for squares]\label{sq-turn}
For $m_1$ and $m_2$ powers of $2$, and arbitrary $n \leq m_1m_2$, in the unlabeled configuration space $\sq(m_1m_2 - n; m_1, m_2)$ the configuration with holes in the first $n$ positions in row-major order and the configuration with holes in the first $n$ positions in column-major order have $L^\infty$ intrinsic distance that grows as $O(m_1 + m_2)$ independent of $n$.
\end{sublemma}

\begin{proof}
We start in column-major order and give an algorithm for changing to row-major order.  We use induction on $m_1 + m_2$.  We use different strategies for when the holes make up at most half the grid and for when they make up more than half the grid.  Of course we could just swap the roles of pieces and holes, but this trick will not work in the hole-turning sublemma for hexagons (Sublemma~\ref{hex-turn}).

We divide the $m_1$ by $m_2$ board into $\frac{m_1}{2}$ by $\frac{m_2}{2}$ quadrants.  In the case where there are $n \leq \frac{1}{2}m_1m_2$ holes, we use at most $m_2$ time steps to push the lower-left-quadrant holes into the lower-right quadrant, $\frac{m_2}{2}$ steps to the right.  Then we use the inductive hypothesis to change each quadrant into row-major order.  Then we use at most $m_1$ time steps to push the lower-right-quadrant holes into the upper-right quadrant, $\frac{m_1}{2}$ steps up.  If it takes at most $C\cdot \left(\frac{m_1}{2} + \frac{m_2}{2}\right)$ time steps to turn each quadrant, then the number of time steps needed for the whole process is at most 
\[m_2 + C \cdot \left( \frac{m_1}{2} + \frac{m_2}{2} \right) + m_1 = \left( 1+ \frac{1}{2}C\right) \cdot (m_1 + m_2),\]
which is at most $C \cdot (m_1 + m_2)$ as long as we choose $C \geq 2$.

The case where there are $n > \frac{1}{2}m_1m_2$ holes is slightly more complicated, and is depicted in Figure~\ref{fig-sq-turn}.  First we swap the configurations in the upper two quadrants, in the following way.  Suppose the starting configuration has $k$ holes in the upper-right quadrant, in column-major order, and $\frac{1}{4}m_1m_2$ holes filling the upper-left quadrant.  We fix the first $k$ holes (in column-major order) in the upper-left quadrant; there are $\frac{1}{4}m_1m_2$ other holes in the upper half of the board, and we push these to the right until they fill the upper-right quadrant.  This takes at most $m_2$ time steps.  Then we use the inductive hypothesis to change each quadrant to row-major order.  Then we push all holes upward as far as possible, using at most $m_1$ time steps.  The result is row-major order, and again the total time is at most $\left(1 + \frac{1}{2}C\right) \cdot (m_1 + m_2) \leq C\cdot (m_1 + m_2)$ steps.
\end{proof}

\begin{figure}
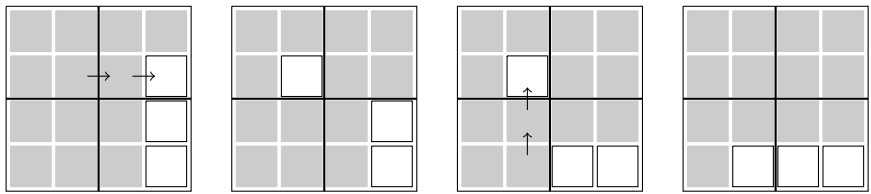

\begin{center}
\SqTurn
\end{center}
\caption{The hole-turning method is to start in column-major order, push some holes right, turn each quadrant, and push all holes up into row-major order.}\label{fig-sq-turn}
\end{figure}

\begin{proof}[Proof of hole-moving lemma for squares (Lemma~\ref{sq-hole})]
We start with an arbitrary configuration of holes, and give an algorithm for changing to the configuration where all the holes precede all the pieces in row-major order.  The process is depicted in Figure~\ref{fig-sq-move}.  We use induction on $m$.  Using the inductive hypothesis (with columns and rows swapped), we put each $\frac{m}{2}$ by $\frac{m}{2}$ quadrant into column-major order.  In column-major order there may be several full columns of holes followed by at most one partial column, which we call the short column.  We combine the two upper quadrants and combine the two lower quadrants as follows: push down the holes in the short column of the right quadrant, then push all holes left, then push up the holes in the short column.  Then we apply the hole-turning sublemma (Sublemma~\ref{sq-turn}) to the upper and lower half-boards to convert them into row-major order.  Then we combine: push right the holes in the short row of the lower half, then push all holes up, then push left the holes in the short row.  If it takes time at most $C \cdot \frac{m}{2}$ to put each quadrant into column-major order, and time at most $C' \cdot m$ for the other steps, then the total time is at most $C \cdot m$ as long as we choose $C \geq 2C'$.
\end{proof}

\begin{figure}
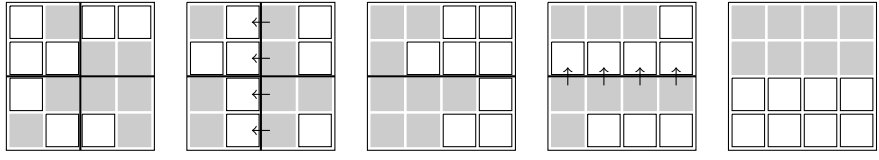

\begin{center}
\SqMove
\end{center}
\caption{The hole-moving method is to put each quadrant into column-major order, merge top quadrants and merge bottom quadrants, turn top and bottom half-boards, and merge top and bottom into row-major order.}\label{fig-sq-move}
\end{figure}

Next we prove the hole-moving lemma for hexagons.

\begin{lemma}[Hole-moving lemma for hexagons]\label{hex-hole}
For $m$ a power of $2$, and arbitrary $n \leq m^2$, the large connected component of the unlabeled configuration space $\hex(n; m)$ has $L^\infty$ intrinsic diameter that grows as $O(m)$ independent of $n$.
\end{lemma}

The proof again relies on a hole-pushing sublemma (Sublemma~\ref{hex-push}) and a hole-turning sublemma (Sublemma~\ref{hex-turn}).  For hexagons the hole-pushing sublemma is more involved than what is needed for squares.  It says that we can push the holes in two half-boards together.

\begin{sublemma}[Hole-pushing sublemma for hexagons]\label{hex-push}
For $m_1$ and $m_2$ divisible by $2$, and $n$ satisfying $3 \leq n \leq m_1m_2$, in the unlabeled configuration space $\hex(m_1m_2 - n; m_1, m_2)$ we consider the configurations in which the holes in the left $m_1$ by $\frac{m_2}{2}$ half-board occupy the first positions in column-major order and the remaining holes occupy the first positions in column-major order of the right half-board.  The $L^\infty$ intrinsic distance from any of these configurations to the one where the holes occupy the first $n$ positions in column-major order has an $O(m_1 + m_2)$ upper bound independent of $n$.
\end{sublemma}

\begin{proof}
We modify the strategy used for squares: push down the short column of holes in the right half-board, push everything left, then push up the short column.  First we consider the case where the right half-board has at least one full column of holes.  In this case pushing down the short column proceeds just as in the hole-pushing sublemma for squares (Sublemma~\ref{sq-push}).  To push left, we divide the board into $2$ by $m_2$ ribbons.  Within each ribbon, we pair up the holes by column and walk each pair to the left along the ribbon, with each pair playing the role of a square hole in the hole-pushing sublemma for squares; if there is only one hole in the right-most column, we group it with the previous column and move that group of $3$ to the left together, somewhat slower than the groups of $2$ but still in $O(m_2)$ steps.  After pushing left, we push up the short column again as in the hole-pushing sublemma for squares.

Next we consider the case where the right half-board has less than a full column of holes.  If there is just one hole in the right half, we send a pair of holes from the left half to walk over and retrieve it.  Otherwise, to push that column down, we need to use the column to the right of it as well, with a zigzag motion as in Figure~\ref{fig-hex-zigzag}.  We group the holes into adjacent pairs and triples.  We walk each group one position to the right, so that the entire column of holes is shifted right, and then walk each grouping one position down and to the left; this process pushes the column of holes down one space from its original position.  Using this strategy, we push the column of holes all the way down.  In the same groupings, we push all the way to the left.  Then we push the short column up, either directly or in a zigzag depending on whether there are holes to the left of it or not.
\end{proof}

\begin{figure}
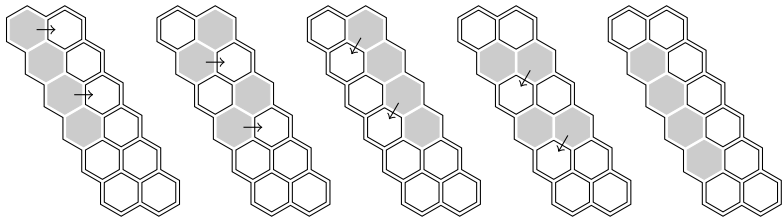

\begin{center}
\HexZigzag
\end{center}
\caption{Zigzagging between two columns, an arbitrarily long column of holes can advance up to $m_1$ positions downward in $O(m_1)$ time steps.}\label{fig-hex-zigzag}
\end{figure}

\begin{sublemma}[Hole-turning sublemma for hexagons]\label{hex-turn}
For $m_1$ and $m_2$ powers of $2$, and arbitrary $n \leq m_1m_2$, in the unlabeled configuration space $\hex(m_1m_2 - n; m_1, m_2)$ the configuration with holes in the first $n$ positions in row-major order and the configuration with holes in the first $n$ positions in column-major order have $L^\infty$ intrinsic distance that grows as $O(m_1 + m_2)$ independent of $n$.
\end{sublemma}

\begin{proof}
We modify the proof of the hole-turning sublemma for squares (Sublemma~\ref{sq-turn}).  There, the method consists of three steps: horizontal pushing, turning each quadrant, and vertical pushing.  For hexagons, for the horizontal pushing step we apply the hole-pushing sublemma for hexagons (Sublemma~\ref{hex-push}), but in reverse so that the holes become separated rather than combined; for the vertical pushing step we apply hole-pushing in the forward direction, but with rows and columns swapped.  Otherwise the proof is the same as in the square case.
\end{proof}

\begin{proof}[Proof of hole-moving lemma for hexagons (Lemma~\ref{hex-hole})]
We modify the proof of the hole-moving lemma for squares (Lemma~\ref{sq-hole}).  There, the method consists of four steps: moving each quadrant to column-major order, horizontal pushing, turning upper and lower half-boards, and vertical pushing.  For hexagons we need to add a preliminary step of getting one pair of adjacent holes into each quadrant.  To do this, we walk our original pair over to a third hole, and carry it over to a fourth hole, and so on, until we have four pairs of adjacent holes.  (If the total number of holes is less than $8$, we can carry them all directly to row-major order and be done.)  Then we move one pair of holes to each quadrant; this whole process can be done in $O(m)$ time steps.  For the remaining steps of the proof for squares, we apply the hexagonal analogues: the inductive hypothesis, hole-pushing (Sublemma~\ref{hex-push}), hole-turning (Sublemma~\ref{hex-turn}), and hole-pushing again.  The result is an algorithm to move the holes to precede all the pieces in row-major order, using $O(m)$ time steps.
\end{proof}

Having proved the hole-moving lemma, we are ready to prove the main theorem for hexagons.

\begin{proof}[Proof of main theorem for hexagons (Theorem~\ref{hex-diam})]
The proof is very similar to the proof of Theorem~\ref{sq-diam}.  We fix $r$ a power of $2$ large enough that $\alpha < \frac{r^2 - 3}{r^2}$, and view the $m$ by $m$ board as an $\frac{m}{r}$ by $\frac{m}{r}$ grid of $r$ by $r$ blocks.  We may assume that the number of holes is exactly three times the number of blocks.

First we use the hole-moving lemma (Lemma~\ref{hex-hole}) to move the holes so that each block has three holes, adjacent to each other.  Then we apply Thompson and Kung's sorting algorithm (Theorem~\ref{grid-swaps}), using Proposition~\ref{hex-connect} to sort the pair of adjacent blocks in each comparison step.  This process allows us to sort the pieces in $O(m)$ time steps so that the blocks are in snake-like row major order and the pieces in each block are in snake-like row major order with the three holes at the end.
\end{proof}

\section{Conclusion}

When exploring configuration spaces such as these, one goal is to find evidence of a phase transition: does the diameter of $\widetilde{\sq}(n; m)$ or $\widetilde{\hex}(n; m)$ grow differently when the density $\frac{n}{m^2}$ is low, compared to when the density is high?  The main results of this paper state that for each density $\alpha \in (0, 1)$ the quantities $\frac{1}{m}\diam(\widetilde{\sq}(\alpha m^2; m))$ and $\frac{1}{m} \diam(\widetilde{\hex}(\alpha m^2; m))$ are bounded as $m$ goes to infinity.  If each one converges to a limit, then that limit is a function of $\alpha$, and for any value of $\alpha$ at which this function or its derivatives are discontinuous, we can say that the diameter has a phase transition at that density $\alpha$.  One could speculate on which densities are good candidates for a phase transition, such as the density at which there exist configurations where every piece is adjacent to a hole.  It would also be useful to estimate the limits of $\frac{1}{m}\diam(\widetilde{\sq}(\alpha m^2; m))$ and $\frac{1}{m} \diam(\widetilde{\hex}(\alpha m^2; m))$, if they exist, in terms of $\alpha$.  Put differently, can we estimate the diameter of $\widetilde{\sq}(n; m)$ or $\widetilde{\hex}(n; m)$ as a function of both $n$ and $m$, up to a constant factor?

The purpose of studying the square and hexagon configuration spaces was to gather clues about disk configuration spaces.  Do the main theorems of this paper have an analogue that describes the diameters of disk configuration spaces or their connected components?  And, for any further results on whether the square and hexagon problems show a phase transition in diameter as we vary density, do those statements apply to disks as well?

This work suggests many other questions.  For instance, does the main theorem for squares (Theorem~\ref{sq-diam}) also apply to square grids in higher dimensions?  More generally, to what extent does Thompson and Kung's result (Theorem~\ref{grid-swaps}) generalize to arbitrary graphs?  Can we guarantee a sorting time in terms of the diameter, say, and some other properties of the graph?  Specifically, given a graph on $n$ vertices, we can construct another graph on the $n!$ permutations of the vertices.  Two permutations are adjacent in the new graph if there exists a set of disjoint edges in the original graph such that we can get from one permutation to the other by swapping the two ends of each of the chosen edges.  Can we estimate, up to a constant factor, the diameter of the new graph?

Finally, we could ask about properties of the configuration spaces other than diameter.  The notion of adjacency used in defining the $L^\infty$ intrinsic metric turns the square and hexagon configuration spaces into graphs.  Instead of diameter, we could look at other graph-theoretic properties such as expansion.  Another approach is to add higher-dimensional cells so as to approximate the disk configuration spaces topologically, and keep track of topological invariants such as Betti numbers; Matthew Kahle and Robert MacPherson (personal communication) have constructed such a cell complex for the square configuration spaces.  Any of these directions for inquiry might give new understanding of how the configuration spaces change shape according to the shape of the board, the shape of the pieces, and the density of the pieces on the board.

\bibliography{fifteen}{}
\bibliographystyle{amsalpha}
\end{document}